\documentclass[draft, 11pt, twoside]{scrartcl}
\usepackage{myscrartcl}%
\usepackage{tikz}
\usepackage[latin1]{inputenc}
\usepackage[english]{babel}

\usepackage{enumerate, enumitem}
\usepackage{amsmath}
\usepackage{amsthm}
\usepackage{amsfonts}
\usepackage{amssymb}
\usepackage{euscript}

\usepackage{fancyhdr}
\newtheorem{theorem}{Theorem}[section]
\newtheorem{lemma}[theorem]{Lemma}

\newtheorem{proposition}[theorem]{Proposition}

\newtheorem{definition}[theorem]{Definition}
\newtheorem{assumption}{Assumption} %%% [section]

\newtheorem{remark}[theorem]{Remark}

\newcommand{\lS} {\ell_{{\mathcal S}}}

\newcommand{\ep}{\varepsilon}
\renewcommand{\phi}{\varphi}
\newcommand{\lref}{\widehat\ell}

\newcommand{\up} {\upsilon}
\newcommand{\wt}{\widetilde}
\newcommand{\ext}{{\rm ext}}

\def\wh{\widehat} %
\def\la{\lambda}%%
\newcommand\wtau[1]{\widehat\tau_{#1}}%%
\newcommand{\wu}{\widehat{\upsilon}}
\def\wla{\widehat\la}%%
\def\wxi{\widehat\xi}%%
\def\wxu{\widehat x_1}%
\def\wxd{\widehat x_2}%
\def\wxf{\widehat x_f}%
\def\a{\alpha} %
\renewcommand{\b}{c} %
\newcommand{\gtause}{\gamma\se_{\wtau{2}}}
\def\e{\varepsilon} %
\def\wh{\widehat} %

\newcommand{\wtc}{\widetilde c}%%
\newcommand{\wtg}{\widetilde\gamma}%%

\renewcommand\L{\mathbb L} %
\newcommand\R{\mathbb R} %
\newcommand\Rn{{\mathbb R}^n} %

\newcommand\id{{\operatorname{id}}} %
\newcommand{\wS}[1]{\widehat S_{#1}} %
\newcommand{\wSinv}[1]{\widehat S^{-1}_{#1}} %
\newcommand{\qo}{\text{a.e. }} %
\newcommand{\cO}{{\mathcal O}} %
\newcommand{\cS}{{\mathcal S}} %
\newcommand{\cU}{{\mathcal U}} %
\newcommand{\cV}{{\mathcal V}} %
\newcommand{\cW}{{\mathcal W}} %
\newcommand{\cinf}{C^\infty} %
\newcommand{\unoforma}{{\boldsymbol s}} %
\newcommand{\scal}[2]{\langle {#1} \, , \, {#2} \rangle} %
\newcommand{\dueforma}[2]{{\boldsymbol\sigma}\left( {#1}, {#2} \right) } %
\newcommand{\liebr}[2]{\left[  {#1}, {#2} \right] } %
\newcommand{\poissonbr}[2]{\left\{  {#1}, {#2} \right\} } %
\newcommand{\liede}[3]{ L_{#1}{#2} \left( {#3}\right) } %
\newcommand{\liededo}[3]{ L^2_{#1}{#2} \left( {#3}\right) } %

\newcommand{\liedede}[4]{ L_{#1}L_{#2}{#3}\left( {#4}\right)  } %
\newcommand{\bsi}{{\boldsymbol\sigma} } %
\newcommand{\ud}{\operatorname{d}\!} %
\newcommand{\uD}{\operatorname{D}\!} %
\newcommand{\cF}{{\mathcal F}} %
\newcommand{\cH}{{\mathcal H}} %
\newcommand{\vFref}[1]{\overrightarrow{\wh F_{#1}}} %
\newcommand{\fref}[1]{{\wh f_{#1}}} %
\newcommand{\Fref}[1]{{\wh F_{#1}}} %
\newcommand{\vF}[1]{\overrightarrow{F_{#1}}} %
\newcommand{\vH}[1]{\overrightarrow{H_{#1}}} %f
\newcommand{\vHmax}[1]{\overrightarrow{H}^{\rm max}} %
\newcommand{\ddt}{\displaystyle\frac{\ud}{\ud t}} %
\newcommand{\dds}{\displaystyle\frac{\ud\phantom{s}}{\ud s}} %
\newcommand{\ddtt}{\displaystyle\frac{\ud^2}{\ud t^2}} %
\newcommand{\lo}{{\lref_0}} %
\newcommand{\lu}{{\lref_1}} %
\newcommand{\ld}{{\lref_2}} %
\newcommand{\lf}{{\lref_f}} %
\newcommand{\dl}{{\delta\ell}} %
\newcommand{\de}{{\delta e}} %
\newcommand{\dep}{\delta p} %
\newcommand{\dx}{{\delta x}} %
\newcommand{\dz}{\delta z} %
\newcommand{\se} {^{\prime \prime}} %
\newcommand{\co}{\overline{\rm{co}}}
\newcommand{\fS}{f_1}
\newcommand{\FS}{F_1}
\newcommand{\vFS}{\overrightarrow{F_1}}

\usepackage{fixme}
\usepackage
[
    lmargin=30mm,
    rmargin=38mm,
    marginparwidth=1in,
    marginparsep=0.2in
]{geometry}

\usepackage{lipsum}
\begin{document}

\title{Strong local optimality for a bang-bang-singular extremal: \\
the fixed-free case}

\author{Laura Poggiolini and Gianna Stefani}

\date{Dipartimento di Matematica e Informatica "Ulisse Dini"\\ Universit\`a degli Studi di Firenze \\[3mm]
{\normalsize{\tt laura.poggiolini@unifi.it},  {\tt gianna.stefani@unifi.it}}}
\maketitle

\begin{abstract}
  In this paper we give sufficient conditions for a Pontryagin extremal trajectory, consisting of two bang arcs followed by a singular one, to be a strong local minimizer for a Mayer problem. The problem is defined on a manifold $M$ and the end-points constraints are of fixed-free type. We use a Hamiltonian approach and its connection with the second order conditions in the form of an accessory problem on the tangent space to $M$ at the final point of the trajectory. Two examples are proposed.
\end{abstract}
%
% \keywords{sufficient conditions, singular control, second variation,
% Hamiltonian methods.}  \bodymatter
% ~\\~\\
%%%%%%%%%%%%%%%%%%%%%%%%%%%%%%%%%
% \newpage
%%%%%%%%%%%%%%%%
\section{Introduction}
In this paper we consider a reference trajectory consisting of two bang arcs followed by a singular (or partially singular) one, for a Mayer problem with fixed final time $T$ and a control affine dynamics.

We give sufficient optimality conditions for the reference trajectory to be a strong
local minimiser in the case when the end-point constraints are of fixed-free type.

A Bolza problem can be reduced to a Mayer one, hence sufficient optimality conditions can be also derived for a Bolza problem, see the examples in Section  \ref{sec:examples}.

Control affine systems can be modelled in different ways; since we want to consider both bang-bang arcs and partially singular arcs, we model the system as follows.

Let $M$ be a finite dimensional manifold and let $X_1, \ldots, X_m $ be smooth vector fields on $M$.
Let $\Delta := \left\{ 
u= \left(u_1, \ldots, u_m \right) \in \R^m \colon u_i \geq 0 ,\  i=1, \ldots, m, \ \sum_{i=1}^m u_i = 1
\right\} $ 
so that at each point $x \in M$ the closed convex hull $\mathcal{X}$ of the vector fields  $X_1, \ldots, X_m$  is given by
\[
\mathcal{X}(x) =  \left\{
\sum_{i=1}^m u_i X_i(x) \colon u = \left(u_1, \ldots, u_m \right) \in \Delta
\right\}.
\]
Let $T > 0$ and $x_0 \in M$, we consider an optimal control problem of the following kind
\begin{subequations}\label{eq:problema}
\begin{align}
& \text{minimize } \ c(\xi(T)) \ \text{subject to } \\
& \dot\xi(t) \in \mathcal{X}(\xi(t))
\quad \qo t \in [0,T], \label{eq:dinamica}\\
& \xi(0) = x_0. \label{eq:inizio}
\end{align}
\end{subequations}
Equivalently, by Filippov's theorem, see e.g.~\cite{BP07}, equation \eqref{eq:dinamica} can also be written as 
\[
 \dot\xi(t) = \sum_{i=1}^m \up_i(t) X_i(\xi(t)), \quad \qo t \in [0,T], \quad \up \in L^\infty\left([0, T], \Delta \right).
 \]
Our aim is to give sufficient conditions for an extremal reference trajectory to be indeed a {\em strong local} optimiser of the problem in the following sense:
\begin{definition}
The trajectory $\wxi \colon [0, T] \to M$ is a {\em strong local minimiser} of problem \eqref{eq:problema} if there exists a neighbourhood $\cU$ of its graph in $\R \times M$ such that $\wxi$ is a minimiser among the admissible trajectories whose graph is in $\cU$, i.e.~among the admissible trajectories which are in a neighborhood of $\wxi$ with respect to the $C^0$ topology.
\end{definition}
Here we assume that the control associated to the reference trajectory is the concatenation of two bang arcs and of a partially singular one, as explained below.
\begin{remark}
In this paper we consider the case when the final point is not constrained, in order to avoid some technical difficulties. 
In a future paper, \cite{PS17}, we shall extend the result to the case when the final point $\xi(T)$ is constrained to a smooth submanifold $N$ of $M$. The extension can be obtained by adding a penalty term and taking advantage of some classical results on quadratic forms due to Hestenes, see \cite{Hes51}, which permit to reduce the problem to a problem with free final point.

In \cite{PS17} we shall also give an explicit formulation of the sufficient conditions for a Bolza problem.
\end{remark}

Assume $\wh\xi$ is the reference trajectory and that there exist times $\wtau{1}, \wtau{2}$, $0  < \wtau{1} < \wtau{2} < T$, 
vector fields $h_1$, $h_2$, $h_3 \in \left\{ X_1, \ldots X_m\right\}$,
 (where $h_1$ and $h_3$ might be the same vector field) and a measurable 
function $\wu \in L^\infty \left( [\wtau{2}, T], (0,1) \right)$ 
such that the solution $\wxi$ to 
\begin{alignat}{2}
& \dot\xi(t) = h_1(\xi(t)) && t \in [0, \wtau{1}), \notag \\
& \dot\xi(t) = h_2(\xi(t)) && t \in (\wtau{1}, \wtau{2}), \notag \\
& \dot\xi(t) = \wu (t)h_3(\xi(t)) +\left( 1 -  \wu (t) \right) h_2(\xi(t))   \qquad && \qo t \in (\wtau{2}, T], \notag \\
&\xi(0) = x_0, \notag 
\end{alignat}
satisfies Pontryagin Maximum Principle (PMP). % and the generalised Legendre-Clebsch condition (GLCC).

Setting $\fS := h_3 - h_2 $ we can write the dynamic on the singular arc as 
\begin{equation}\label{eq:terzo}
\dot\xi(t) = h_2(\xi(t)) + \wu (t) \fS(\xi(t)), \qquad t \in (\wtau{2}, T).
\end{equation}
 %%%%
We shall also define the time-dependent reference vector field $\wh f_t$ as
\begin{equation}
\wh f_t := \begin{cases}
h_1 & t \in [0, \wtau{1}), \\
h_2 & t \in (\wtau{1}, \wtau{2}), \\
h_2 + \wu(t) \fS \quad & \qo t \in (\wtau{2}, T].
\end{cases}
\label{eq:reffield}
\end{equation}
To get the sufficient conditions we use a Hamiltonian approach and its connection with the second order conditions, whose leading ideas are the following:  
\begin{enumerate}
\item To use the symplectic properties of the cotangent bundle to compare the costs of neighbouring admissible trajectories by lifting them to the cotangent bundle. 
\item To define a suitable Hamiltonian flow  $\cH_t$ in the cotangent bundle $T^*M$,  emanating from a horizontal Lagrangian submanifold $\Lambda$.
Since the final point is free, the flow is considered to have the final time  $T$ as a starting time and to go backward in time, Sec. \ref{sec:overflow}.
\item To obtain a suitable second order approximation (\emph{$2^{nd}$ variation}) in the form of a coordinate-free linear-quadratic (LQ) problem and to require its coercivity,
Sec. \ref{sec:2ndvar}.
\item To show that the derivative of $\cH_t$ along the reference extremal is, up to an isomorphism, the linear Hamiltonian flow associated to the LQ problem, see
Sec. \ref{sec:consequence}  and \ref{sec:antiso}.
\item To deduce that the projection on $M$ of $\cH_{t}$ emanating from $\Lambda$ is locally invertible (see Sec. \ref{sec:proof}), so that we can go back to the first issue and we can compare the costs of neighbouring admissible trajectories by lifting them to the cotangent bundle, Theorem \ref{thm:main1}. 
\end{enumerate}
In this paper we only give the main ideas of the constructions and some proofs of the main results, while all the details will be given in \cite{PS17}.

\section{Notation and preliminaries}
In this paper we use some basic element of the theory of symplectic
manifolds referred to the cotangent bundle $T^*M.$ For a general
introduction see \cite{Arn80}, for specific application to Control
Theory we refer to \cite{AS04}.  Let us recall some
basic facts and let us introduce some specific notations.

Denote by $ \pi \colon T^*M \to M $ the canonical projection, for
$\ell \in T^*M$ the space $T_{\pi\ell}^*M$ is canonically embedded in
$T_{\ell}T^*M$ as the space of tangent vectors to the fibres.

The canonical Liouville one--form $\unoforma$ on $T^*M$ and the
associated canonical symplectic two--form $\bsi = \ud\unoforma$ allow
associating to any, possibly time--dependent, smooth Hamiltonian $H_t
\colon T^*M \rightarrow \R$, a Hamiltonian vector field $\vH{t}$, by
\[
\bsi(v,\vH{t}(\ell))=\scal{ \ud H_t(\ell)}{v} ,\quad \forall v\in
T_{\ell}T^*M .
\]
In this paper we consider all the flows -- both in $M$ and in $T^*M$ -- as
starting at the final time $T$, unless otherwise explicitly stated.  We
denote the flow of $\vH{t}$ from time $T$ to time $t$ by
\[
\cH :(t,\ell)\mapsto \cH(t,\ell)=\cH_t(\ell) .
\]
We keep these notation throughout the paper, namely the overhead arrow
denotes the vector field associated to a Hamiltonian and the script
letter denotes its flow from time $T$, unless otherwise stated.

Finally we recall that any vector field $f$ on the manifold $M$
defines, by lifting to the cotangent bundle, a Hamiltonian
\begin{equation*}
  F \colon \ell\in T^*M  \mapsto \scal{\ell}{f(\pi\ell)} \in \R.
\end{equation*}
We denote by $\FS$, $H_i$ the Hamiltonians associated to $\fS$, $h_i$, $i=1, \ 2, \ 3$, respectively and by
\[
\begin{alignedat}{2}
& H_{ij} := \poissonbr{H_{i}}{H_{j}}, \quad && i, \ j \in 1, 2, 3 \\
& H_{ijk} := \poissonbr{H_{i}}{\poissonbr{H_{j}}{H_{k}}}, \quad && i, \ j, \ k \in 1, 2, 3 \\
\end{alignedat}
\]
the Poisson parenthesis and the iterated Poisson parenthesis between Hamiltonians. We recall that $H_{ij}$ is the Hamiltonian associated to the Lie bracket $h_{ij}:=
[h_i, h_j]$.

In order to write the second order variation of the problem in an useful way we shall consider flows going backwards in time, i.e.~starting at the final time $T$.
The flow from time $T$ of the reference vector field $\wh f_t$
is a map defined in a neighbourhood of $ \wxf := \wh\xi(T)$.  We
denote such flow as $ \wh S_t \colon M \to M$, $ t \in [0, T] $, i.e.~
\[
\ddt \wS{t}(x) = \wh f_t \circ \wS{t}(x), \quad \wS{T}(x) = x  .
\]
We also denote 
$
% \wxz := \wxi(0) = \wS{0}(\wxf), \qquad 
\wxu :=\wxi(\wtau{1}) = \wS{\wtau{1}}(\wxf), \ \wxd :=\wxi(\wtau{2}) = \wS{\wtau{2}}(\wxf). $

The time-dependent Hamiltonian associated to $\wh f_t$ is denoted by $\wh F_{t}$ and its flow backwards in time starting at time $T$ is denoted by $\wh\cF_t$.

Also we use the following notation from differential geometry: $\liede{f}{\a}{\cdot}$ is the Lie derivative of a function $\a$ with respect to the
vector field $f$. Moreover, if $G$ is a $C^1$ map from a manifold $M_1$
in a manifold $M_2$, we denote its tangent map at a point $x \in M_1$ as
$T_x G$. If the point $x$ is clear from the context, we also write
$T_x G= G_* \, $.

\subsection{The necessary conditions}
We start by stating the necessary conditions of optimality, i.e.~Pontryagin Maximum Principle (PMP) and the Legendre condition.
Since there is no constraint on the final point, then PMP must hold in its normal form: 
\begin{assumption}[PMP]\label{ass:PMP}
There exists a map $\wla \colon [0, T] \to T^*M$, which is absolutely continuous and 
such that
\begin{alignat*}{2}
& \pi\wla(t) = \wxi(t) \qquad & t \in[0,T] , \\
& \dot\wla(t) = \vFref{t}(\wla(t)) \qquad & \qo t \in[0,T] , \\
& \wla(T) = - \ud c(\wxf), \\
& % \scal{\wla(t)}{\wh f_t(\wxi(t))} = 
\Fref{t}(\wla(t)) = \max\left\{
\scal{\wla(t)}{v} \colon v \in \mathcal{X}(\wxi(t)) 
\right\}  \qquad & t \in[0,T] .
%% & \L(\wla(t)) := \scal{\wla(t)}{\liebr{\fS}{\liebr{h_2}{\fS}}(\wxi(t))} \geq 0 \qquad & t \in[\wtau{2}, T] .
\end{alignat*}
\end{assumption}
We shall use the following notation for the end points and for the switching points of $\wla(t)$:
\[
  \lf := \wla(T), \quad  \ld := \wla(\wtau{2}) = \wh\cF_{\wtau{2}}(\lf),  \quad   \lu := \wla(\wtau{1}) = \wh\cF_{\wtau{1}}(\lf), \quad 
\lo := \wla(0) = \wh\cF_0(\lf).
\]
Thanks to the structure of the reference trajectory, PMP gives the following necessary conditions:
\begin{enumerate}
 \item On the first bang arc, $ t \in[0,\wtau{1}]$, we get $ \  H_1(\wla(t)) \geq \scal{\wla(t)}{X},  \quad \forall X \in \mathcal{X}(\wxi(t))$.
 %% \label{eq:max1}
\item On the second bang arc, $ t \in[\wtau{1}, \wtau{2}]$, we get $\ H_2(\wla(t)) \geq \scal{\wla(t)}{X}, \quad \forall X \in \mathcal{X}(\wxi(t))$, in particular $H_1(\ld) = H_2(\ld)$. %% \label{eq:max1}
\item On the singular arc, $t \in [\wtau{2}, T]$, we get 
\[
 \left( H_2 + a \FS\right) (\wla(t)) \geq \scal{\wla(t)}{X}, \quad \forall X \in \mathcal{X}(\wxi(t)),  \ \forall a \in [0,1] ,
 \]
which implies $\FS(\wla(t)) \equiv 0$ and, by differentiation, 
\[
\ddt \FS(\wla(t)) = 
H_{23} (\wla(t)) 
%% = \poissonbr{H_2}{H_3}\circ\wla(t)
\equiv \poissonbr{H_2}{\FS}\circ\wla(t) 
 = 0
\] 
and 
\begin{equation} \label{eq:sing3}
 - H_{232}(\wla(t))+ \wu(t)\L(\wla(t))  =  0 % \\
% %& \poissonbr{H_2}{\poissonbr{H_2}{H_3}}(\wla(t))
% %+ \wu(t) \poissonbr{H_3 - H_2}{\poissonbr{H_2}{H_3}}(\wla(t))  = \\
% & \poissonbr{H_2}{\poissonbr{H_2}{\FS}}(\wla(t))
% + \wu(t) \poissonbr{\FS}{\poissonbr{H_2}{\FS}}(\wla(t))  = \\
% & \scal{\wla(t)}{\liebr{h_2}{\liebr{h_2}{\fS}}(\wxi(t))} 
% + \wu(t) \scal{\wla(t)}{\liebr{\fS}{\liebr{h_2}{\fS}}(\wxi(t))}  
\end{equation}
where 
\[
\L(\ell) := (H_{323} + H_{232})(\ell)= \scal{\ell}{\liebr{\fS}{\liebr{h_2}{\fS}}(\pi\ell)} , \quad  \ell \in T^*M.
\]
\item At the first switching time $\wtau{1}$ we get 
$H_{12}(\lu) = \left. \ddt \left(H_2 - H_1 \right)\circ\wla(t) \right\vert_{t = \wtau{1}} \geq 0$, 
see for example \cite{ASZ02b}.
\item At the second switching time $\wtau{2}$ we get
$H_{232}(\ld) = \left. - \, \ddtt  F_1 \circ \wla(t)\right\vert_{t = \wtau{2}^-} \!\! \geq 0$, see \cite{PS11}.
\end{enumerate}
Moreover, other necessary conditions are known, namely the Goh condition (which in this case is automatically satisfied) and the {\em generalised Legendre condition} (GLC), see e.g.~\cite{AS04}, 
\[
R(t) := \L(\wla(t))  \geq 0 \qquad t \in[\wtau{2}, T] .
\]

\section{Assumptions and main result}
\subsection{Regularity conditions}
We now state regularity conditions by requiring strict inequalities to hold whenever necessary conditions yield mild inequalities.
\begin{assumption}[Regularity along the bang arcs]\label{ass:regbang}
\begin{equation*} %% \label{eq:bang1}
\begin{split}
& H_1(\wla(t)) > \scal{\wla(t)}{X}, \qquad \forall X \in \mathcal{X}(\wxi(t)) \setminus \{h_1(\wxi(t)) \}, \quad \forall t \in [0, \wtau{1}) , \\ %
& H_2(\wla(t)) > \scal{\wla(t)}{X}, \qquad \forall X \in \mathcal{X}(\wxi(t)) \setminus \{h_2(\wxi(t)) \}, \quad \forall t \in (\wtau{1}, \wtau{2}) ,
\end{split}
\end{equation*}
\end{assumption}
i.e.~we require that the reference control is the only maximising control along the given arc.
\begin{assumption}[Regularity along the singular arc]\label{ass:regsing}
%For $x \in M$, let 
%\[
%V_{23}(x) := \left\{ (1-s)h_2(x) +s h_3(x) \colon s \in [0,1]\right\} = \left\{ h_2(x) + s f_1(x) \colon s \in [0,1]\right\}. 
%\]
For any $a, s \in [0,1]$ and any  $t \in[\wtau{2}, T]$ 
\begin{equation*}
H_2(\wla(t)) + a F_1(\wla(t)) > \scal{\wla(t)}{X(\wxi(t))}, \qquad \forall X \in \mathcal{X} , \quad  X \neq h_2 + s f_1,
\end{equation*} 
\end{assumption}
i.e.~we require that the set of maximisers along the singular arc is the edge defined by $h_2$ and $h_3$ .

\begin{assumption}[Regularity at the switching points]\label{ass:switch}
\begin{equation}
 H_{12}(\lu) > 0, \qquad H_{232}(\ld)  > 0. \label{eq:switch2} 
\end{equation}
\end{assumption}
\begin{assumption}[Strong generalised Legendre condition]\label{ass:SGLC}
\begin{equation}
 R(t) = \L(\wla(t)) =  \poissonbr{F_1}{\poissonbr{H_2}{F_1}}(\wla(t)) 
> 0 \qquad t \in[\wtau{2}, T] \label{eq:SGLC}\tag{SGLC}
\end{equation}
\end{assumption}
Thanks to \eqref{eq:SGLC} we can recover the value of the control along the singular arc:
\begin{equation*}
\wu(t) = \dfrac{ H_{232}}{\L}(\wla(t)) \qquad \forall t \in (\wtau{2}, T],
\end{equation*}
so that, by recurrence, one can easily prove that $\wu \in C^\infty([\wtau{2}, T], (0,1))$.  

 Notice that under  \eqref{eq:SGLC}, the second inequality in \eqref{eq:switch2} is equivalent to the discontinuity of the reference vector field at $t = \wtau{2}$. 

 For $\ell$ in a neighborhood of the range of the singular arc $\wla([\wtau{2}, T])$ in $T^*M$ we can define the {\em Hamiltonian feedback control} 
\begin{equation}
u_\cS(\ell) := \dfrac{H_{232}}{\L} (\ell).
\label{eq:feedback}
\end{equation} 
Notice that $\wla$ also satisfies the autonomous differential equation
\begin{equation}
\label{eq:auto}
\dot\la(t) = \left( \vH{2} + \overrightarrow{ u_\cS F_1} \right)(\la(t)) .
\end{equation} 
The condition $\wu(t) \in (0,1)$  reads
\begin{equation}
H_{232}(\wla(t)) > 0, \quad H_{323}(\wla(t))> 0 \qquad \forall t \in (\wtau{2}, T]. \label{eq:segni}
\end{equation}
%Thus the reference vector field is discontinuous at $\wxi(\wtau{2})$ if and only if $\wu(\wtau{2}^+) > 0$, i.e.~if and only if the regularity condition at $\wtau{2}$, equation \eqref{eq:switch2} holds.
%%%%%%%%%%%%%%%%%%%%%%%%%%%%%%%%%%%%%%%%%%%%%%%%%%%%%%%%%%%%
\subsection{The extended second variation}
\label{sec:2ndvar}
%In order to apply the Hamiltonian approach to prove sufficient conditions to
%strong optimality, first we extend in a suitable way the cost
%function, namely the smooth function $p_0\,c$
%is defined on  $N$, but it can be
%extended to the whole manifold $M$ in such a way that the
%transversality conditions hold on
%the whole tangent space. Let $ \b \colon M \rightarrow \R$ be a function such that
%\begin{equation*}
%\beta =p_0\, c \quad   \text{on } N,  \qquad \wla(T) = \ud\,(-\beta)(\wxf) \quad \text{on } T_{\wxf} M. %% \label{eq:beta}
%  \end{equation*}
%In the normal case $(p_0 = 1)$ $\beta$ is a cost function
%equivalent to the original one while in the abnormal case $(p_0 = 0)$
%it is an extension of the zero function. When $p_0 = 0$ all the costs
%disappear, we study a problem with a zero cost and indeed we are
%studying the constraints.  Proving that $\wxi$ is a {\em strict}
%strong minimizer with $p_0 = 0$ implies that it is isolated among admissible trajectories. 
The sufficient conditions will be derived
studying a sub problem of the given one. 
%\begin{equation}
%  \text{Minimize } \ \beta(\xi(T))   
%\ \text{ subject to \eqref{eq:dinamica}--\eqref{eq:fine}}, \label{eq:betapb}
%\end{equation}
Namely we consider problem \eqref{eq:problema}, the reference vector field $\fref{t}$  and allow only for perturbations of $\wu$ on the singular interval $\left(\wtau{2}, T\right)$ and for perturbations of the switching time $\wtau{1}$. Following the ideas of  \cite{PS11} the subproblem can be written as
%\begin{subequations}\label{eq:subpb}
%\begin{align}
%& \text{Minimize }  c(\xi(T)) \ \text{ subject to} \\
%& \dot\xi(t) = \begin{cases}
%h_1(\xi(t)) \quad & t \in (0, \tau_1) , \\
%h_2(\xi(t)) \quad & t \in (\tau_1, \wtau{2}) , \\
%h_2(\xi(t)) + u(t)f_1(\xi(t)) \quad & t \in (\wtau{2}, T) ,
%\end{cases} \\
%& \tau_1 \in (0, \wtau{2}), \quad u \in (0, 1), \\
%& \xi(0) = x_0 .
%\end{align}
%\end{subequations}
%
%
\begin{subequations}\label{eq:subpbu0}
\begin{align}
& \text{Minimize }  \b(\xi(T)) \ \text{ subject to} \\
& \dot\xi(t) = \begin{cases}
\up_0(t) h_1(\xi(t)) \quad & t \in (0, \wtau{1}) , \\
\up_0(t) h_2(\xi(t)) \quad & t \in (\wtau{1}, \wtau{2}) , \\
h_2(\xi(t)) + \up(t)f_1(\xi(t)) \quad & t \in (\wtau{2}, T) ,
\end{cases} \\
& \up_0(t) > 0, \ \int_0^{\wtau{2}} \up_0(t)\ud t = \wtau{2}, \quad \up(t) \in (0, 1), \\
& \xi(0) = x_0 .
\end{align}
\end{subequations}
%We write both the second variation and the extended second variation following the ideas of \cite{ASZ98a} and \cite{PS11} respectively, but in order to gain useful information we take advantage of the push-forward to the final time $T$ instead of the pull-back to the initial time $t=0$. A detailed construction will appear in \cite{PS17}.

Set
\begin{equation}
g_{t} :=  \wSinv{t*} f_1 \circ \wS{t} \, ,\ t\in [ \wtau{2},T] , \quad 
k_i :=  \wSinv{\wtau{1} *} h_i \circ \wS{\wtau{1}}, \ i=1,2, \quad k := k_1 - k_2,
\label{eq:pushvf}
\end{equation}
i.e.~$g_t$  is the push-forward of $\fS$ from time $t \in [\wtau{2}, T]$ to time $T$ while the $k_i$-s are the push-forward of the $h_i$-s from the first switching time $\wtau{1}$ to $T$.
With this notation the second variation of \eqref{eq:subpbu0} is given by 
\begin{equation}
%\begin{split}
%& 
J\se  [(\dx, \delta\up_0(\cdot),  \delta\up(\cdot))]^2 =  
\int_{\wtau{2}}^{T}  \delta\up(t) \liedede{\delta\eta(t)}{g_{t}}{\b}{\wxf} \, dt  + \dfrac{\e_0^2}{2} \left( \liededo{k}{\b}{\wxf} + H_{12}(\lref_{1}) \right)
\label{eq:secvarfina}
\end{equation}
subject to
\begin{equation*} %%\label{eq:eta}
\dot{\delta\eta}(t) =  \delta\up(t) g_{t}(\wxf),\qquad
\delta\eta(T) = \dx \in T_{\wxf}M, \qquad \delta\eta(\wtau{2}) = \e_0 k (\wxf)
\end{equation*}
where
\[
\e_0=\int_0^{\wtau{1}}\delta \up_0 (t)\, dt =-  \int_{\wtau{1}}^{\wtau{2}}\delta \up_0 (t)\, dt.
\]
The precise construction will appear in \cite{PS17}.  We point out that the perturbation at the switching time $\wtau{1}$ gives rise to a cost in the accessory problem.

 We then  extend the second variation to a new quadratic form called {\em extended second variation}. 
 Following the same lines as in the appendix of \cite{PS11} and setting
\[
w(t) := \displaystyle\int_{t}^{\wtau{2}} \delta u(s) \ud s, \qquad
\e_1 := w(T),
\] 
the extended second variation of \eqref{eq:subpbu0} is given by the following singular LQ problem on the interval $[ \wtau{2},T]$.
%$$J= \b(x)+\wh\a(\eta(\wtau{2},x))=\gamma(x)+\int_T^{\wtau{2}}L_{\eta(t)}\wh\a(\eta(t)\, dt$$
%e poi ho fatto l'approssimazione del 2ndo ordine essendo il primo ordine uguale a zero per PMP 
%\begin{equation}\label{eq:deltaeta}
%2J\se [(\dx, \delta u)]^2 = 
%\frac{1}{2}\gamma\se[\dx]^2 +\frac{1}{2}\int^{\wtau{2}}_{T}  2\,\delta u(t)\, L_{\delta\eta(t)} L_{g_{t}}\wh \a(\wxf)\, dt 
%\end{equation}
%subject to
%\begin{equation}
%\dot{\delta\eta}(t) = \delta u\, g_{t}(\wxf),\qquad
%\delta\eta(\wtau{2}) = \delta y,\qquad \zeta(T)=\dx\in T_{\wxf}\wh N .
%\end{equation}
%Integrating by parts (Gho trasformation) we get
%the quadratic form 
\begin{equation} \label{eq:secvarfin}
\begin{split}
& J\se_\ext [(\dx, \e_0, \e_1 ,w)]^2 =  
- \e_1 \liedede{\dx}{f_1}{c}{\wxf} 
- \frac{\e_1^2}{2} \liededo{f_1}{c}{\wxf} 
+ \\
& + \dfrac{\e_0^2}{2} \left( \liededo{k}{c}{\wxf}  
+ H_{12}(\lu)\right) + \frac{1}{2}\int_{\wtau{2}}^T  \left(2\, w(t) \liedede{\zeta(t)}{\dot g_{t}}{c}{\wxf} + w(t)^2 R(t)  \right)  \, \ud t 
\end{split}
\end{equation}
subject to
\begin{equation}
\dot{\zeta}(t) =  w(t) \dot g_{t}(\wxf),\qquad
\zeta(\wtau{2}) =\e_0 \, k(\wxf),\qquad \zeta(T)=\dx +\e_1 f_1(\wxf) .   \label{eq:zeta}
\end{equation}
This means that we consider the quadratic form $J\se_\ext$ defined by \eqref{eq:secvarfin} on the linear space called {\em space of admissible variations} given by 
\begin{multline*}
\qquad \cW_\ext  := \{ (\dx, \e_0,\e_1, w)  \in T_{\wxf}M \times \R \times \R \times L^2([\wtau{2}, T]) \colon \\
\text{system \eqref{eq:zeta} admits a solution} \}.
\qquad
 %% \label{eq:spazio} 
\end{multline*}
Notice that 
\begin{equation}
\dot g_{t} =\wSinv{t*} h_{23} \circ \wS{t}, \quad  t \in [ \wtau{2},T] .\label{eq:dotg} 
\end{equation}

Choosing $(\dx, \e_0, \e_1, w(\cdot)) = (-\fS(\wxf), 0, 1, 0)$ in \eqref{eq:secvarfin} we get $\liededo{\fS}{\b}{\wxf} > 0$ as a necessary condition for the coercivity of the extended second variation \eqref{eq:secvarfin} on $\cW_\ext$.

Let $\cO(\wxf)$ be a neighborhood of $\wxf$ in $M$ and consider the set
\[
\wt M := \left\{ x \in \cO(\wxf) \colon \liede{\fS}{\b}{x} = 0\right\} .
\]
If $\liededo{\fS}{\b}{\wxf} > 0$, then $\wt M$ is a hypersurface such that
\[
T_{\wxf}\wt M   =  \left\{
\dz \in T_{\wxf}M \colon \liedede{\dz}{\fS}{\b}{\wxf} = 0
\right\}.
\]
For $x = \exp(r\fS)(z)$, $z \in \wt M$ set
\[
\wtc(x) := \b(z),
\]
i.e.~we extend $\left. c \right\vert_{\wt M}$ as a constant function along the integral lines of $f_1$. If $\cO(\wxf)$ is sufficiently small, then the function $\wtc \colon \cO(\wxf) \to \R$ is smooth and it enjoys the following properties
\begin{equation}
\begin{alignedat}{2}
& \wtc(\wxf) = c(\wxf), \quad && \ud\wtc(\wxf) = \ud c(\wxf),  \\ %%
& \wtc(x) \leq c(x), \quad && \liede{f_1}{\wtc}{x} = 0 \quad \forall x \in \cO(\wxf).
\end{alignedat}
\label{eq:ctilde}
\end{equation}
Following \cite{PS11} it can be shown that  the coercivity of \eqref{eq:secvarfin} on $\cW_\ext$ is equivalent to $\liededo{f_1}{c}{\wxf} > 0$ plus the coercivity of 
\begin{equation}
\label{eq:Jtilde}
\begin{split}
\widetilde J_\ext [(\dx,\e_0,w)]^2 =& 
\dfrac{\e_0^2}{2}  \left( \liededo{k}{\wtc}{\wxf} + H_{12}(\lu)\right) + \\
&  + \frac{1}{2} \int_{\wtau{2}}^{T}   \left( %
 2\, w(t) \liedede{\zeta(t)}{\dot g_{t}}{\wtc}{\wxf}  + R(t) w(t) ^2  %
\right)  \ud t
\end{split}
\end{equation}
subject to 
\begin{equation}\label{eq:ridotte2}
\dot{\zeta}(t) =  w(t) \dot g_{t}(\wxf),\quad
\zeta(\wtau{2}) =\e_0 \, k(\wxf),\quad \zeta(T)=\dx \in  T_{\wxf}M.
\end{equation}
In the case when $\liede{f_1}{c}{\cdot} \equiv 0$ in $\cO(\wxf)$,  we set  $\wtc := c $. Thus, also in this case, we end up with \eqref{eq:Jtilde} subject to \eqref{eq:ridotte2}.
%
%In any case we end up with the LQ form on a subspace of $N\se \times \R \times L^2([\wtau{2}, T], \R)$ 
%\begin{multline}
%\qquad 
%\widetilde J_\ext [(\dx,\e_0,w)]^2 =
%\dfrac{\e_0^2}{2}  \left( \liededo{k}{\wtc}{\wxf} + H_{12}(\lu)\right)  +\\
% + \frac{1}{2} \int_{\wtau{2}}^{T}  \left( %
% 2\, w(t) \liedede{\zeta(t)}{\dot g_{t}}{\wtc}{\wxf}  + R(t) w(t) ^2  %
%\right)  \ud t 
%\qquad \label{eq:secvarexta}
%\end{multline}
%in 
%\[
%\wt\cW_\ext = \left\{
%(\dx, \e_0, w) \in N\se \times \R \times L^2([\wtau{2}, T], \R) \colon \text{system } \eqref{eq:ridotte2} \text{ admits a solution}
%\right\}
%\]
\begin{assumption}\label{ass:coerc}
We assume the following conditions hold
\begin{enumerate}
\item The quadratic form  $\widetilde J_\ext$, \eqref{eq:Jtilde}, is coercive on 
\begin{multline*}
\wt\cW_\ext := \{ (\dx, \ep_0,w) \in T_{\wxf}M   \times \R  \times L^2([\wtau{2}, T], \R) \colon
\\
\text{system \eqref{eq:ridotte2} admits a solution} \} .
%% \label{eq:tildeW}
\end{multline*}
\item Either $\liededo{f_1}{c}{\wxf} > 0$ or $\liede{f_1}{c}{\cdot} \equiv 0$ in a neighborhood $\cO(\wxf)$ of $\wxf$ in $M$.
\end{enumerate}
\end{assumption}

\subsection{The main result}
We can now state the main result of this paper
\begin{theorem}
\label{thm:main}
Let $\wxi$ be the admissible trajectory defined in \eqref{eq:terzo}. Assume that $\wxi$ satisfies Assumptions \ref{ass:PMP}--\ref{ass:coerc}. Then $\wxi$ is a strict strong local optimal trajectory of \eqref{eq:problema}.
\end{theorem}
More precisely we prove that Assumptions \ref{ass:PMP}-\ref{ass:SGLC} plus 1.~of Assumption \ref{ass:coerc} imply that $\wxi$ is a strict strong locally optimal trajectory for the cost $\wtc(\xi(T))$. This concludes the proof in the case   $\liededo{\fS}{c}{\cdot} \equiv 0$.

When $\liededo{\fS}{c}{\wxf} > 0$, $c=\wtc$ on $\wt M$, hence we have to compute the difference  $c-\wtc$ along the integral lines of $f_1$ starting at $z\in\wt M$, so that \eqref{eq:ctilde} easily gives the claim. 
\section{Hamiltonian approach}
The first step in applying the Hamiltonian approach described in the Introduction, is the construction of an overmaximised Hamiltonian flow.
Indeed the presence of a singular arc prevents us from using the maximized Hamiltonian (see \cite{PS11}) which can be used in the classical case, i.e.~when it is $C^2$, see \cite{AS04}.
The overmaximized Hamiltonian was introduced in \cite{Ste07} and then used in \cite{PS11, PS13b} . In \cite{SZ16} the authors give a sistematic extension of the classical techniques to the 
case of an overmaximized Hamiltonian whose flow is only Lipshitz continuous.
\subsection{The overmaximised flow}\label{sec:overflow}
The \eqref{eq:SGLC} condition (Assumption \ref{ass:SGLC}) implies that there exists a neighborhood $\cO_{\rm s}$ of the range of the singular arc $\wla([\wtau{2}, T])$ in $T^*M$ such that
\begin{equation*}
\Sigma := \left\{
\ell \in \cO_{\rm s} \colon F_1(\ell) = 0
\right\} = \left\{
\ell \in \cO_{\rm s} \colon H_2(\ell) = H_3(\ell)
\right\} 
%%\label{eq:Sigma}
\end{equation*}
and
\begin{equation*}
\cS := \left\{
\ell \in \Sigma \colon H_{23}(\ell) = 0
\right\} = \left\{
\ell \in \cO_{\rm s} \colon H_2(\ell) = H_3(\ell), \ H_{23}(\ell) = 0
\right\} %% \label{eq:Esse}
\end{equation*}
are smooth simply connected manifolds of codimension $1$ and $2$, respectively. More precisely $\vH{23}$ is transverse to $\Sigma$ in $\cO_{\rm s}$, while $\vF{1}$ is tangent to $\Sigma$ and transverse to $\cS$, see \cite{PS11}.  %%%%

Here we want to describe how the regularity conditions allow to define in a tubular neighborhood $\cO$ of the graph of $\wla$ in $[0, T] \times T^*M$, a time-dependent Hamiltonian function $H \colon \cO \to \R$ whose flow satisfies the assumptions stated in \cite{SZ16}. The coercivity of the second variation will then guarantee the invertibility of the projected overmaximised flow of such Hamiltonian. %

In \cite{PS11} the authors prove that possibly restricting $\cO_{\rm s}$, the following implicit function problem has a solution $\theta \colon \cO_{\rm s} \to \R$:
\begin{equation*} %% \label{eq:theta}
\theta(\ell) \colon
\begin{cases}
H_{23} \circ \exp \theta\vF{1}(\ell) = 0, \\
\theta(\ell) = 0 \quad \text{if } H_{23}(\ell) = 0,
\end{cases}
\end{equation*}
%Since $\left. 
%\derpar{\theta} H_{23} \circ \exp \theta\vF{1}(\ell)\right\vert_{(0,\ell)} = \dueforma{\vF{1}}{\vH{23}}(\ell) = \L(\ell) > 0$ for any $\ell \in \wla([\wtau{2}, T])$, the implicit function theorem applies.
and 
\[
\scal{\ud\theta(\lS)}{\dl} = \dfrac{- \, \dueforma{\dl}{\vH{23}(\lS)}}{\L(\lS)}
\qquad \forall \lS \colon H_{23}(\lS) = 0.
\]
Let  
\[
\wt{H}_2(\ell) := H_2 \circ \exp \theta(\ell)\vF{1}(\ell).
\]
From the results in \cite{PS11} we can derive the following Lemma:
\begin{lemma}\label{le:invarianze}
Possibly restricting $\cO_{\rm s}$ the following properties hold
\begin{enumerate}
\item $\wt{H}_2(\ell) \geq H_2(\ell)$ for any $\ell \in \Sigma$.
Equality holds if and only if $\ell \in \cS$.
\item For any $ \lS \in \cS$
\[
\uD\left( \wt{H}_2 - H_2 \right)(\lS) = 0 , \quad 
\uD^2 \left( \wt{H}_2 - H_2 \right)(\lS) = \dfrac{\left(\dueforma{\dl}{\vH{23}(\lS)}\right)^2}{\L(\lS)}.
\]
 \item $\overrightarrow{\wt H_2}$ and hence $\overrightarrow{\wt H_2} + \wu(t)\vFS$ are tangent to $\Sigma$ for any $t \in [\wtau{2}, T]$.
\end{enumerate}
Set
\begin{equation}\label{eq:Htdef}
H_t(\ell) = \wt H_2 + \wu(t)\FS\qquad \forall (t, \ell) \in [\wtau{2}, T] \times \cO_{\rm s}
\end{equation}
\begin{enumerate}
\item[4. ] $\left. \wla \right\vert_{[\wtau{2}, T]}$ is the solution of the Cauchy problem
\[
\dot\la(t) = \vH{t}(\la(t)), \quad \la(T) = \lf.
\]
\end{enumerate}
Moreover the following invariant properties hold:
\begin{enumerate}
\item[5.] $\vH{2}$ is invariant with respect to the flow of $\vH{t} = \overrightarrow{\wt{H}}_{2} + \wu(t)\vF{1}$ along the singular arc of the reference trajectory:
\[
\vH{2}(\wla(t)) = \cH_{t*}\vH{2}(\lf) \qquad t \in [\wtau{2}, T];
\]
\item[6.] $\vF{1}$ is invariant on $\Sigma$ with respect to the flow of $\vH{t}$:
\[
\vF{1}\circ\cH_t(\ell) = \cH_{t*}\vF{1}(\ell), \quad \forall \ell \in \Sigma, \quad t \in [\wtau{2}, T].
\]
\end{enumerate}
\end{lemma}
This Lemma is the main tool for handling the singular arc. The bang arcs present a different kind of problems. Namely we need to define the {\em switching times} near the reference switching points $\lu$ and $\ld$ of the Pontryagin extremal $\wla$. 
In \cite{PS11} it is shown that the flow of $\vH{2}$ is the maximised one in a left hand side neighborhood of $\wtau{2}$ {\em if and only if} $H_{23}(\ell) \geq 0$.
In order to overcome this problem we introduce a { \em correction } of the backwards flow from time $\wtau{2}$ by {\em keeping the flow  on $\Sigma$ when $H_{23}(\ell) <0$}. 

By the implicit function theorem applied to the problem: 
\begin{equation*} %%\label{eq:tau2}
 \begin{cases}
        H_{23} \circ \exp( t_2 - \wtau{2})\overrightarrow{\widetilde{H}}_{2}\left( \ell \right) = 0 , \\ %%
        t_2(\ell) = \wtau{2} \quad \text{ if }H_{23}(\ell) = 0 .
      \end{cases}
    \end{equation*}
it is possible to define a function $t_2 \colon \cO(\ld) \to \R$ such that if $\ell \in \Sigma$, then $t_2(\ell) = \wtau{2}$ if and only if $\ell \in \cS$; moreover 
%As  $\left. 
%\derpar{t_2} 
%H_{23} \circ \exp( t_2 - \wtau{2})\overrightarrow{\wt{H}}_{2}\left( \ell \right) 
%\right\vert_{(0,\ld)} \!\!\!\!\! = \dueforma{\overrightarrow{\wt{H}}_{2}}{\vH{23}}(\ld) = \dueforma{\vH{2}}{\vH{23}}(\ld) = H_{223}(\ld)$, by the regularity assumption at the second switching time \eqref{eq:switch2}, the implicit function theorem applies and the function $t_2(\ell)$ is well defined for any $\ell$ in a neighborhood of $\ld$ in $T^*M$. 
\[
\scal{\ud t_{2}(\ld)}{\dl} = \dfrac{- \, \dueforma{\dl}{\vH{23}(\ld)}}{H_{223}(\ld)}.
\]
We set
\begin{equation*}
\tau_2(\ell) := \min \left\{ t_2(\ell), \wtau{2} \right\} = \begin{cases}
t_2(\ell) \quad & \text{if } H_{23}(\ell) < 0 ,  \\
\wtau{2} & \text{if } H_{23}(\ell) \geq 0 .
\end{cases}
%% \label{eq:tau2}
\end{equation*}
The next step will be the definition of the switching time $\tau_1 \colon \cO(\ld) \to \R$.  
Actually, the implicit function theorem applies also to
\begin{equation*} %%\label{eq:tau1}
\begin{cases}
\left( H_2 - H_1 \right) \circ \exp \left(\tau_1 -\tau_{2}(\ell) \right)\vH{2} \circ\exp \left(\tau_{2}(\ell) - \wtau{2} \right)\overrightarrow{\wt H_{2}}(\ell) = 0, \\
\tau_1(\ld) = \wtau{1},
\end{cases}
\end{equation*}
see \cite{ASZ02b} and 
\begin{equation}
\scal{\ud\tau_{1}(\ld)}{\dl} = \dfrac{- \, \dueforma{\exp\left(\wtau{1} -\wtau{2} \right)\vH{2}_{\, *}\dl}{\left( \vH{2} - \vH{1}\right)(\lu)}}{H_{12}(\lu)}. \label{eq:dtau1}
\end{equation}
%%%%%%%%%%%%%%%%%%%%%%%%%%%%%%%%%%%%%%%%%%%%%%%%%%%%%%%%%%%%%%%
We can now define the flow $(t, \ell) \mapsto \cH_t(\ell)$ backwards in time emanating from a neighborhood $\cO(\lf)$ of $\lf$ in $T^*M$ at time $T$.
%Namely, for any $t \in [0, \wtau{2}]$, let
%\begin{align*}
%& \wt\ell = \exp(\wtau{2} - t)\overrightarrow{\wt H}_2(\ell), \\
%& \ell_2 =  \exp(\wtau{2} - \tau_2(\wt\ell))\overrightarrow{\wt H}_2 \circ \exp(t - \tau_2(\wt\ell))\vH{2} (\ell), 
%\end{align*}
%it is the Hamiltonian flow associated to the Hamiltonian
%\begin{equation}
%\label{eq:Ht}
%\begin{split}
%&H_t(\ell) = H(t, \ell) :=  \\
%& \begin{cases}
%\wt{H}_2(\ell) + \wu(t)F_1(\ell),  \quad & (t, \ell) \in [\wtau{2}, T] \times \cO_{\rm s}, \\
%\wt{H}_2(\ell) ,  \qquad & (t, \ell) \in [\tau_2(\wt\ell), \wtau{2}) \times \cO(\ld), \\
%%\qquad & \text{where } \wt\ell= \exp(\wtau{2} - t)\overrightarrow{\wt H}_2(\ell) \\
%H_2(\ell),    \qquad & (t, \ell) \in [\tau_1(\ell), \tau_2(\ell)) \times \cO(\ld), \\
%H_1(\ell),   \qquad & (t, \ell) \in [0, \tau_1(\ell)) \times \cO(\ld).
%\end{cases}
%\end{split}
%\end{equation}
Namely, for any $t \in [\wtau{2}, T]$, $\cH_t(\ell)$ is the flow  associated to the time-dependent Hamiltonian defined in \eqref{eq:Htdef}.
Let $\wt\ell := \cH_{\wtau{2}}(\ell)$. For $t < \wtau{2}$, $\cH_t(\ell)$ is defined as
\begin{equation}
\cH_t(\ell) := 
\begin{cases}
\exp (t - \wtau{2})\overrightarrow{\wt H_2}(\wt\ell) \quad & t \in [\tau_2(\wt\ell), \wtau{2}], \\
\exp (t - \tau_{2}(\wt\ell))\vH{2}\circ\cH_{\tau_{2}(\wt\ell)}(\ell) \quad & t \in [\tau_1(\wt\ell), \tau_2(\wt\ell)), \\
\exp (t - \tau_{1}(\wt\ell))\vH{1}\circ\cH_{\tau_{1}(\wt\ell)}(\ell) \quad & t \in [0, \tau_1(\wt\ell)),
\end{cases}\label{eq:flusso}
\end{equation}
see Figure \ref{fig:Hami}.
\begin{figure}\label{fig:Hami}
\begin{tikzpicture}%[>=stealth]
	\pgftransformscale{.8} %
    \definecolor{gray1}{gray}{0.95} %
    \definecolor{gray2}{gray}{0.825} %
    \definecolor{gray3}{gray}{0.7} %
    \definecolor{gray4}{gray}{0.675} %
    \filldraw[gray4](0,  -2.8)--(4.5, -2.8)--(6.5, 2.8)--(0, 2.8); %
		\node[anchor=south] at (2.5, 0.5 ) {$H_1$}; %
    \filldraw[gray2](4.5,  -2.8)--(9, -2.8)--(10.5, 0)--(10.5, 2.8)--(6.5,2.8); %
		\node[anchor=south] at (8, 0.5 ) {$H_2$}; %
    \filldraw[gray3](9, -2.8)--(10.5, -2.8)--(10.5, 0); %
		\node[anchor=north] at (9.75, -2 ) {$\wt H_2$}; %
    \filldraw[gray1](10.5, -2.8)--(15, -2.8)--(15, 2.8)--(10.5, 2.8); % 
		\node[anchor=south] at (13, 0.5 ) {$\wt H_2 + \wh\nu(t)F_1$ }; %
	\draw[thick](4.5, -2.8)--(6.5, 2.8) node [midway, sloped, above]{$t =\tau_{1}(\wt\ell)$}; %
	\draw[thick](5.5, -3)--(5.5, 3);% 	
		\node[anchor=south] at (5.5, 3 ) {$t = \wtau{1}$}; 
	\draw[thick](10.5, -3)--(10.5, 3); %
		\node[anchor=south] at (10.5, 3 ) {$t = \wtau{2}$}; 
	\draw[thick](15, -3)--(15, 3); %
			\node[anchor=south] at (15, 3 ) {$t = T$}; 
	\draw[thick](9, -2.8)--(10.5, 0); %
	\draw[thick, dashed](10.5, 0)--(11.25, 1.4); %
	\draw (9,-2.8)--(10.5,0)node [midway, sloped, above]{$t =\tau_{2}(\wt\ell)$};
	\draw [<->, very thick] (0,3.8) node (yaxis) [above] {$T^*M$} |-
        (16,0) node (xaxis) [right] {$t$}; %
    \draw[very thick](0,0)--(0, -2.8); %
\end{tikzpicture}
  \caption{The over--maximised Hamiltonian}
\end{figure}
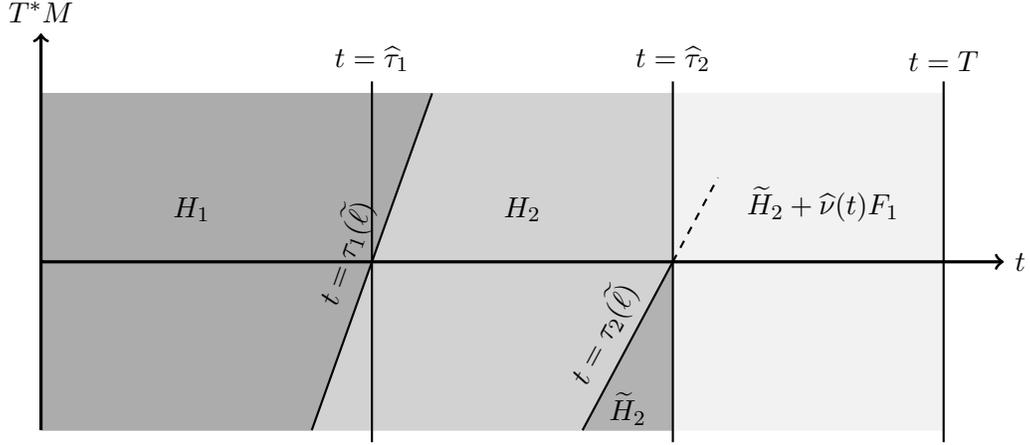
%%%%%%%%
\begin{remark}
Notice that $\cH$ is $\cinf$ on $[\wtau{2}^+, T] \times \cO(\lf)$ and it is  Lipschitz continuous on $[0, \wtau{2}^-] \times \cO(\lf)$. 
%
%For any $t < \tau_{2}(\wt\ell)$,  $\cH_t$ is $C^1$ but at the points $(\tau_1(\wt\ell), \ell)$ where it is Lipschitz continuous.
\end{remark}
We now state and prove the main result obtained by the Hamiltonian approach, see \cite{SZ16}. After, we shall exploit the coercivity of $\wt J$ in order to obtain the required invertibility property.
%%%%%%%%%%%%%%%%%%%%%%%%%%%%%%
\begin{theorem}\label{thm:main1}
Let $\Lambda := \left\{ \ud \, (-\wtc)(x) \colon x \in \cO(\wxf) \right\}$. Assume  the projected overmaximised flow emanating from $\Lambda$ is locally Lipschitz invertible onto a neighborhood $\cU$ of the graph of $\wxi$ in $[0, T] \times M$:
\begin{equation}\label{eq:mappa}
\id \times \pi\cH \colon (t, \ell) \in [0,T] \times \Lambda \mapsto (t, \pi\cH_t(\ell)) \in  \cU.
\end{equation}
Then $\wxi$ is a strict strong locally optimal trajectory for the cost $\wtc(\xi(T))$ subject to \eqref{eq:dinamica}-\eqref{eq:inizio}.
\end{theorem}
\begin{proof}
Clearly $(\id \times \pi\cH)^{-1}(t, \wxi(t)) = (t, \lf)$ for any $t \in [0, T]$.
Let $\xi \colon [0, T] \to M$  be an admissible trajectory for \eqref{eq:problema} whose graph is in $\cU$ and let
\[
(t, \mu(t)) := (\id \times \pi\cH)^{-1}(t, \xi(t)), \qquad
\la(t) := \cH_t(\mu(t)),
 \qquad t \in [0, T].
\]
Let $\varphi \colon  [0,1] \to \Lambda$  be a smooth curve such that $\varphi(0) = \mu(T)$, $\varphi(1) = \lf$. In $[0, T] \times \Lambda$ we can consider the closed path obtained by the concatenation of the curves $t \in [0, T] \mapsto (t, \mu(t))$, $s \in [0,1] \mapsto (T, \varphi(s))$ and of the curve $t \in [0, T] \mapsto (t, \lf)$ ran backwards in time.

Integrating the one-form $\omega := \cH^*\left( p \ud q - H_t \ud t \right)$ (which is exact on $[0, T] \times \Lambda$, see \cite{SZ16}, we obtain 
\begin{equation}
\begin{split}
0 = \oint \omega & = \int_{\id \times \mu}  \scal{\la(t)}{\dot\xi(t)} - H_t(\la(t))  \ud t + \int_\varphi \cH^*p \ud q \\ %%
& - \int_{\id \times \lf}  \scal{\wla(t)}{\dot\wxi(t)} - H_t(\wla(t))  \ud t .
\end{split}\label{eq:circuito1}
\end{equation}
By construction of the overmaximised Hamiltonian $H_t$ the integrand is non positive along $\id \times \mu$ and is identically zero along $\id \times \lf$. Thus
\begin{multline}
\qquad 0 \leq \int_\varphi \cH^*p \ud q = \int_0^1 \scal{\varphi(s)}{\dds (\pi\varphi)(s)} \ud s \\ %%
=  \int_0^1 \scal{\ud\, (-\wtc)(\pi\varphi(s))}{\dds (\pi\varphi)(s)} \ud s = \wtc(\xi(T)) - \wtc(\wxf).  \qquad
\label{eq:circuito2}
\end{multline}
Thus $\wtc(\xi(T)) \geq \wtc(\wxf)$, i.e.~the reference trajectory $\wxi$ is a strong local minimiser for the cost $\wtc$. Let us show that in fact it is a strict one.

If $\wtc(\xi(T)) = \wtc(\wxf)$, then \eqref{eq:circuito1}-\eqref{eq:circuito2} imply that 
\begin{equation}
\scal{\la(t)}{\dot\xi(t)} - H_t(\la(t)) = 0 \qquad \qo t \in [0,T].
\label{eq:strict1}
\end{equation}
Since $\xi(0) = x_0 = \wxi(0)$, we also have $\la(0) = \lo$ and from the regularity condition along the bang arcs, Assumption \ref{ass:regbang}, we easily get $\la(t) = \wla(t)$ for any $t \in [0, \wtau{2}]$, so that $\xi(t) = \pi\la(t) = \pi\wla(t) = \wxi(t)$ for any $t \in [0, \wtau{2}]$. In particular $\la(\wtau{2}) = \ld$.  

Moreover, for $t\in[\wtau{2}, T]$, equation \eqref{eq:strict1} yields $\wt H_2(\la(t)) = H_2(\la(t))$, i.e. $\la(t) \in \cS$. 
%Differentiating with respect to $t$ we thus obtain
%\begin{equation}
%\dueforma{\dot\la(t)}{\vH{23}(\la(t))} = 0 \quad \forall t \in [\wtau{2}, T].
%\label{eq:strict2}
%\end{equation}
Let $\Sigma_{\wxi(t)}$ be the intersection of $\Sigma$ with the fiber over $\wxi(t)$ and consider the function 
\[\Delta \colon \ell \in \Sigma_{\wxi(t)} \mapsto \scal{\ell}{\dot\xi(t)} - H_t(\ell) \in \R.
\] 
By PMP the function $\Delta$ is non positive and by \eqref{eq:strict1} it is null in $\la(t)$. Differentiating $\Delta$ with respect to the vertical fiber we thus obtain
\begin{equation}
\scal{\dep}{\dot\xi(t) - \pi_*\vH{t}(\la(t))} = 0,  \quad %%
\forall \dep \in T^*_{\xi(t)}M \text{, such that } \ \scal{\dep}{\fS(\xi(t))} = 0.
\label{eq:strict3}
\end{equation}
Hence there exists $b(t) \in \R$ such that
\[
\dot\xi(t) = \pi_*\vH{t}(\la(t)) + b(t) \fS(\xi(t)) \qquad \forall t \in [\wtau{2}, T].
\]
Hence, by Lemma \ref{le:invarianze}, point 6,
\[
\dot\mu(t) = \left( \pi\cH_t \right)^{-1}_* \left( \dot\xi(t) - \pi_*\vH{t}(\la(t)) \right) = b(t) \left( \pi\cH_t \right)^{-1}_*\fS(\xi(t)) =
b(t) \vFS(\mu(t)).
\]
Thus
\[
\dot\la(t) = \vH{t}(\la(t)) + \cH_{t *}\dot\mu(t) = \vH{2}(\la(t)) + \left(\wu(t) + b(t) \right) \vFS(\la(t)).
\]
Finally, since $\la(t) \in \cS$, we get
\begin{equation}
0 = \dueforma{\dot\la(t)}{\vH{23}(\la(t))} = - H_{232}(\la(t)) + \left(\wu(t) + b(t) \right)  \L(\la(t)) .
\label{eq:strict4}
\end{equation}
Comparing \eqref{eq:strict4} with \eqref{eq:feedback} we obtain 
\[
\wu(t) + b(t) = u_\cS(\la(t)), 
\]
so that $\la(t)$ and $\wla(t)$ solve the same Cauchy problem on the interval $[\wtau{2}, T]$:
\[
\dot\la = \vH{2}(\la) + u_\cS(\la) \vFS(\la), \qquad \la(\wtau{2}) = \ld.
\]
Hence $\la \equiv \wla$ and $\xi \equiv \wxi$. This proves that $\wxi$ is a strict strong locally optimal trajectory for the cost $\wtc(\xi(T))$.
\end{proof}

\subsection{Consequences of the coercivity of $\wt J$}\label{sec:consequence}
In this section we exploit the coercivity of the second variation, Assumption \ref{ass:coerc} {\em a)}. 
Let $\Lambda := \left\{
\ud \, (-\wtc)(x) \colon x \in \cO(\wxf)
\right\}$.

Assume $k(\wxf) \neq 0$, i.e.~$h_1(\wxu) \neq h_2(\wxu)$. In order to rewrite the extended second variation \eqref{eq:Jtilde} as a standard LQ form, choose $\omega \in T^*_{\wxf}M$ such that $\scal{\omega}{k(\wxf)} = 1$ and set
\[
\gtause := H_{12}(\lu)\omega\otimes\omega - \frac{1}{2}\left( \omega\otimes  \liedede{(\cdot)}{k}{(-\wtc)}{\wxf}
+ \liedede{(\cdot)}{k}{(-\wtc)}{\wxf} \otimes\omega \right) .
\]
We obtain
\begin{equation}
\wt J [(\dx,w)]^2 = \frac{1}{2}\gtause[\zeta(\wtau{2})]^2+\frac{1}{2}\int_{\wtau{2}}^{T}  \left( w(t) ^2 R(t)+
 2\, w(t) \liedede{\zeta(t)}{\dot g_{t}}{(-\wtc)}{\wxf}  \right)  \ud t 
\label{eq:jsebis}
\end{equation}
subject to
\begin{equation}\label{eq:zetabis}
\dot{\zeta}(t) =  w(t) \dot g_{t}(\wxf),\quad
\zeta(\wtau{2}) = 
 \delta y\in\R\, k(\wxf), 
\quad \zeta(T)=\dx\in T_{\wxf}M.
\end{equation}
If $k(\wxf) = 0$ then $\zeta(\wtau{2}) = 0$, so that in \eqref{eq:Jtilde} $\e_0$ is a decoupled variable and $J_{\ext}$ is equivalent to the problem described by \eqref{eq:jsebis}-\eqref{eq:zetabis} whatever the quadratic form $\gtause$ is.

Consider the Lagrangian subspace of trasversality conditions
\[
L\se_T=\left\{(0,\dx)\colon \dx\in  T_{\wxf}M\right\}.
%\, ,\quad L\se_{\wtau{2}}=
%\R\left\{\left(- \gtause[ k(\wxf),\cdot]\, ,\, k(\wxf)\right)\right\}\oplus  \{k(\wxf)\}^\perp
\]
Let
\[
\cW := \left\{
(\dx, w) \in T_{\wxf}M \times L^2([\wtau{2}, T]) \colon \text{system \eqref{eq:zetabis} admits a solution}
\right\}
\]
and consider the subspace of $\cW$
\[
% \begin{split}
\cV  := \left\{\de=(\dx,w)\in \cW \colon \zeta ( \wtau{2}) = 0 \right\}.
%\cW_t & := \left\{\de=(\dx,w)\in \cW \colon w(s)=0 ,\ \forall s\in [ \wtau{2},t]\right\}, \quad t\in [\wtau{2},T].
%\end{split}
\]
Notice that $\cV = \cW$ if and only if $k(\wxf) = 0$. 
It can be easily shown, see \cite{Hes51}, that
\begin{proposition}
$\wt J$ is coercive if and only if $\wt J$ is coercive on $\cV$ and $\wt J [\de]^2>0$ for any $\de\in\cW$, $\de\not= 0$, which is $\wt J$-orthogonal to $\cV$.
\end{proposition}
%%%%%%%%%%%%%
%%\subsection{The singular arc}
%%%%%%%%%%%%%
The Hamiltonian relative to \eqref{eq:jsebis}--\eqref{eq:zetabis} is given by the quadratic form
\begin{equation}
H\se_t(\dep,\dx) = -\frac{1}{2R(t)}\left( \scal{\dep}{\dot g_{t}(\wxf)} + \liedede{\dx}{\dot g_{t}}{(-\wtc)}{\wxf} \right)^2
\label{eq:Hsec}
\end{equation}
while the associated Hamiltonian linear system with initial conditions in $L\se_T$ is given by
\begin{equation}\label{eq:hamisyse}
\left\{
\begin{alignedat}{2}
& \dot{\mu}(t) =   \frac{1}{R(t)}\! \left(\! \scal{\mu(t)}{\dot g_{t}(\wxf)} 
\!
+ \!
\liedede{\zeta(t)}{\dot g_{t}}{(-\wtc)}{\wxf} \!\right)
\liedede{(\cdot)}{\dot g_{t}}{(-\wtc)}{\wxf} , \quad && \mu(T)=0\\
& \dot{\zeta}(t) =  \frac{- 1}{R(t)}\! \left( \! \scal{\mu(t)}{\dot g_{t}(\wxf)}
\!
+ \! 
\liedede{\zeta(t)}{\dot g_{t}}{(-\wtc)}{\wxf} \!  \right)
\dot g_{t}(\wxf), &&  \zeta(T)=\dx .
\end{alignedat}
\right. 
\end{equation}
$\wt J$ is coercive on $\cV$ if and only if for any solution of the Hamiltonian system \eqref{eq:hamisyse} where $\dx\not= 0$,  we have  $\zeta(t)\not =0$ for any $t \in [\wtau{2}, T]$, see for example \cite{SZ97}. %
This concludes the case $k(\wxf) = 0$.

Assume $k(\wxf) \neq 0$ and consider the variations $\de \in \cW $ which are $\wt J$-orthogonal to $\cV$. % 
In terms of system \eqref{eq:hamisyse} the bilinear form associated to $\wt J$ \eqref{eq:jsebis} is given by 
\begin{equation}
\wt J[\de, \overline\de] = \scal{\overline\mu(\wtau{2})}{\zeta(\wtau{2})} +  \scal{\mu(\wtau{2})}{\overline\zeta(\wtau{2})} + \wtg\se_{\wtau{2}}[\zeta(\wtau{2}), \overline\zeta(\wtau{2})] ,
\label{eq:bilin}
\end{equation}
where $\de = (\dx, w)$ and $\overline\de = (\overline\dx, \overline w)$ are in $\cW$ and $\left( \mu(\wtau{2}), \zeta(\wtau{2})\right)$ and  $\left( \overline\mu(\wtau{2}), \overline\zeta(\wtau{2})\right)$ are the solutions of the Hamiltonian system \eqref{eq:hamisyse}  with initial conditions $(0, \dx)$ and $(0, \overline \dx)$, respectively.
Thus $\de \in \cW \cap \cV^{{\wt J}^\perp}$ if and only if there exists $\dep \in T^*_{\wxf}M$ such that 
\[
\wt J[\de, \overline\de] = \scal{\dep}{\overline\zeta(\wtau{2})} \qquad \forall \overline\de \in \cW
\]
i.e.~if and only if
\[
\begin{cases}
\scal{\overline\mu(\wtau{2})}{\zeta(\wtau{2})} = 0 \qquad & \forall \overline\de, \\
\dep = \mu(\wtau{2}) + \wtg\se_{\wtau{2}}[\zeta(\wtau{2}), \, \cdot \, ]. \qquad &
\end{cases}
\]
Hence 
\begin{equation}
0 < \wt J[\de]^2 = \scal{\mu(\wtau{2})}{\zeta(\wtau{2})} + \wtg\se_{\wtau{2}}[\zeta(\wtau{2})]^2 \qquad \forall \de \in  \cW \cap \cV^{{\wt J}^\perp}.
\label{eq:posperp}
\end{equation}

%%%%%%%%%%%%%%%%%%%%

Since $\zeta(\wtau{2}) \in \R \, k(\wxf)$, from equation \eqref{eq:posperp} we get
\begin{equation}
0 <  \wtg\se_{\wtau{2}}[k(\wxf)]^2 + \scal{\mu(\wtau{2})}{k(\wxf)} 
= H_{12}(\lu) - \liededo{k}{(-\wtc)}{\wxf} + \scal{\mu(\wtau{2})}{k(\wxf)} .
\label{eq:posperp2}
\end{equation}
%%%%%%%%%%%%%%%%%%%%%%%%%%%
\subsection{The antisymplectic isomorphism}\label{sec:antiso}
Define the linear mapping $\iota$ by
\[
\iota\colon (\dep,\dx)\in T^*_{\wxf}M \otimes T_{\wxf}M \mapsto \dl := -\dep + \ud\,(-\wtc)_*\dx\in T_{\lf}T^*M
\]
so that
\[
\iota^{-1} \colon \dl\in T_{\lf}T^*M \mapsto \left( \ud\,(-\wtc)_*\pi_*\dl - \dl , \pi_*\dl \right) \in T^*_{\wxf}M\otimes T_{\wxf}M.
\]
Moreover $\iota$ is an antisymplectic ismorphism, i.e.
\[
\dueforma{\iota (\dep,\dx)}{\iota (\overline{\dep},\overline{\dx})} = \dueforma{(\overline{\dep},\overline{\dx})}{(\dep,\dx)},  \quad \forall (\dep,\dx) , \ (\overline{\dep},\overline{\dx}) \in T^*_{\wxf}M \otimes T_{\wxf}M.
\]
With this notation we get
%\begin{align*}
%& 
\begin{equation*}
\iota  L\se_T  =\left\{\ud\,(-\wtc)_*\dx \colon \dx\in T_{\wxf}M \right\} = T_{\lf}\Lambda .
%, \\
%& \iota  L\se_{\wtau{2}}  = \R\left(\gamma\se_{\wtau{2}}[k(\wxf),(\cdot)]+ \ud\,(-\wtc)_*k(\wxf)]\right)
%\oplus \{ k(\wxf)\}^\perp  \\
%&  = \R\left( \left( H_{12}(\lu)+\frac{1}{2}\liededo{k}{\wtc}{\wxf} \right)\omega +\frac{1}{2} \liedede{(\cdot)}{k}{\wtc}{\wxf} + \ud\,(-\wtc)_*k(\wxf)\right)
%\oplus \{ k(\wxf)\}^\perp.
%\end{align*}
\end{equation*}
%%%%
Following the lines of Lemma 9 in \cite{PS11} one can prove the following Lemma:
\begin{lemma}
Let $\cH\se_t$ and $\cH_t$ be the Hamiltonian flows associated to the quadratic Hamiltonian $H\se_t$ defined in \eqref{eq:Hsec} and to the overmaximised Hamiltonian $H_t$ defined in \eqref{eq:Htdef}, respectively. Then 
\begin{equation}
\iota\cH\se_t\iota^{-1}=\wh\cF_{t*}^{-1}\cH_{t*} \qquad \forall t \in [\wtau{2}, T].
\label{eq:iotaflusso}
\end{equation}
\end{lemma}
\subsection{Proof of the main result}\label{sec:proof}
%%%%%%%%%%%%%%
Applying Theorem \ref{thm:main1}, the proof of our main result, Theorem \ref{thm:main}, is completed once we show that $\pi\cH_t$ is locally Lipschitz one-to-one for each $t \in [0, T]$. In fact, as $[0, T]$ is a compact interval,   the map $\id \times \pi\cH_t$ defined in \eqref{eq:mappa} is locally Lipschitz invertible if and only if  for all $t \in [0, T]$ the map
\[
\pi\cH_t \colon \Lambda \mapsto \pi\cH_t(\ell) \in \cO ( \wxi(t))
\]
is locally Lipschitz invertible.
We in fact show that $\pi_*\cH_{t*} \colon T_{\lf}\Lambda \to T_{\wxi(t)}M$  is one-to-one for $t \neq \wtau{1}$ and by means of Clarke inverse function theorem, see \cite{Cla76, Cla83}, for $t=\wtau{1}$.

Since the Hamiltonian $\wh F_t$ is the lift of a vector field, from the coercivity of $\wt J$ on $\cV$ and \eqref{eq:iotaflusso}, the claim holds for any $t \in [\wtau{2}, T]$.
%\[
%\pi\cH_t\colon \Lambda \to \cO(\wxi(t))\, ,\qquad t\in [ \wtau{2},T]
%\]
%is locally invertible. In particular there exists a function $\a_2$ such that 
%\[
%\Lambda_2 := \cH_{\wtau{2}}(\Lambda) = \left\{\ud\a_2(x) \colon x \in \cO(\wxd) \right\} \subset \Sigma.
%\]
%
%
%
%As proved in \cite{PS11}, the linearised projected overmaximised flow at time $t$, $\left(\pi\cH_t\right)_*$, is locally invertible also for $t \in (\wtau{1}, \wtau{2}]$. In order to prove that $\pi\cH_t$ is locally invertible on a common neighborhood $\cO_\Lambda(\lf)$ of $\lf$ in $\Lambda$ for any $t \in [0, \wtau{2}]$ we only need to prove the invertibility of $\left(\pi\cH_{\wtau{1}}\right)_*$ in a neighborhood of $\lf$ in $\Lambda$. Denote by $\fludue{t} \colon \Lambda_2 \to T^*M$ the flow at time $t \in [0, \wtau{2}]$ of $\vH{t}$ emanating from $\Lambda_2$ at time $\wtau{2}$. Then $\cH_{\wtau{1}}$ is locally invertible at $\lf$ if and only if $\fludue{\wtau{1}}$ is locally invertible at $\ld$.
%This property will be proved using Clarke inverse function theorem, see \cite{Cla76, Cla83}. See also \cite{PSp11, PSp15} for an use of such theorem in a similar context.
%
%For any $\ell \in \Lambda_2$ we have
For $t \in (\wtau{1}, \wtau{2})$ from the definition of the flow, equation \eqref{eq:flusso}, we get 
\[
\left( \pi\cH_{t} \right)_* = \exp (t -\wtau{2})h_{2*}\pi_*\cH_{\wtau{2}*};
\]
hence we now have to prove the invertibility of $\left( \pi\cH_{\wtau{1}} \right)_*$.

Let $\dl \in T_{\lf}\Lambda$ and set $\wt\dl := \left( \pi\cH_{\wtau{2}} \right)_*\dl$. Notice that since $\left( \pi\cH_{\wtau{2}} \right)_*$ is one-to-one, then $\pi_*\wt\dl =0$ if and only if $\wt\dl = 0$.

Thus the linearization of $\pi\cH_{\wtau{1}}(\ell) $ at $\lf$ is given by
\begin{align*}
\left( \pi\cH_{\wtau{1}}\right)_* \! \dl &= \!
\begin{cases}
\exp(\wtau{1} - \wtau{2})h_{2\, *}   \pi_* \wt\dl \ & \scal{\ud\tau_{1}(\ld)}{\wt\dl} < 0 , \\
\scal{\ud\tau_{1}(\ld)}{\wt\dl}(h_2 - h_1)(\wxu) +   \exp(\wtau{1} - \wtau{2})h_{2\, *}    \pi_* \wt\dl   \;  & \scal{\ud\tau_{1}(\ld)}{\wt\dl} > 0.
\end{cases} \\
&= \! \begin{cases}
\exp(\wtau{1} - \wtau{2})h_{2\, *}   \pi_* \wt\dl \qquad & \scal{\ud\tau_{1}(\ld)}{\wt\dl} < 0 , \\
\exp(\wtau{1} - \wtau{2})h_{2\, *} \left(
  \pi_* \wt\dl  - \scal{\ud\tau_{1}(\ld)}{\wt\dl}\wt k (\wxd)   \right)  \quad & \scal{\ud\tau_{1}(\ld)}{\wt\dl} > 0.
\end{cases} 
\end{align*}
where $\wt k := \wS{\wtau{2}*} k\circ\wSinv{\wtau{2}} = \exp ((\wtau{2} - \wtau{1})h_2)_* (h_1 - h_2)
\circ \exp (\wtau{1} - \wtau{2})h_2$.
It thus suffices to prove that for any $a \in [0,1]$ and $\dl\in T_{\lf}\Lambda$, $\dl \neq 0$
\[
(1-a)  \pi_*(\wt\dl) + a  \left(  \pi_*(\wt\dl) - \scal{\ud\tau_{1}(\ld)}{\wt\dl}\wt k (\wxd) \right) \neq 0.
\]
If $\scal{\ud\tau_{1}(\ld)}{\wt\dl}\wt k (\wxd) = 0 $ there is nothing to prove. 
Otherwise  assume by contradiction there exist $a \in [0,1]$, $\dl \in T_{\lf}\Lambda$ such that
\begin{equation}
 \pi_* \wt\dl- a \, \scal{\ud\tau_{1}(\ld)}{\wt\dl}\wt k (\wxd)  = 0. \label{eq:invert}
\end{equation}
Since $\left(\pi\cH_{\wtau{2}}\right)_*$ is bijective, there exists a function 
$\a_2 \colon \cO(\wxd) \to \R$ such that 
\begin{equation}\label{eq:alpha2}
\ud\a_2(\wxd) = \ld %%
\text{, and }\ \cH_{\wtau{2}*}\left(T_{\lf}\Lambda \right) = \ud\a_{2*} \left( \pi\cH_{\wtau{2}} \right)_* \left(T_{\lf}\Lambda \right) .
\end{equation}

 Thus, from \eqref{eq:invert} we get
\begin{equation}
0 = \ud\a_{2\, *} \left(  \pi_* \wt\dl - a \, \scal{\ud\tau_{1}(\ld)}{\dl}\wt k (\wxd) \right) 
=   \wt\dl  - a \, \scal{\ud\tau_{1}(\ld)}{\wt\dl} \ud\a_{2\, *} \wt k (\wxd) .
\label{eq:invert2}
\end{equation}
Computing $\ud\tau_1(\ld)$ on each side of \eqref{eq:invert2} we finally get
\[
\begin{split}
0 & =  \scal{\ud\tau_{1}(\ld)}{\wt\dl}  - a \, \scal{\ud\tau_{1}(\ld)}{\wt\dl} \scal{\ud\tau_{1}(\ld)}{\ud\a_{2\, *} \wt k (\wxd)} = \\ %
& 
= \scal{\ud\tau_{1}(\ld)}{\wt\dl} \left( 1 - a \, \scal{\ud\tau_{1}(\ld)}{\ud\a_{2\, *} \wt k (\wxd)} \right) 
\end{split}
\]
i.e.~$1 - a \, \scal{\ud\tau_{1}(\ld)}{\ud\a_{2\, *} \wt k (\wxd)}  = 0$ or, equivalently by \eqref{eq:dtau1}, 
\begin{equation*}
H_{12}(\lu) - a \,  \dueforma{\ud\a_{2\, *}\wt k(\wxd)}{\overrightarrow{\wt K}(\ld)} = 0
\end{equation*}
where $\overrightarrow{\wt K} = \exp \left( \wtau{2} - \wtau{1} \right){\vH{2}_*}\left( \vH{2} - \vH{1}\right)$, so that  
\begin{equation}\label{eq:contra}
H_{12}(\lu) - a \,  \liededo{\wt k}{\a_2}{\wxd} = 0 .
\end{equation}
We now use \eqref{eq:posperp2}, i.e.~the coercivity of $\wt J$, to get a contradiction.
Let $(0, \dx) \in L\se_T$ be such that $\cH\se_{\wtau{2}}(0, \dx) = \left( \mu(\wtau{2}), k(\wxf)\right)$. Then
\[
\begin{split}
 & \scal{\mu(\wtau{2})}{k(\wxf)} = \dueforma{\cH\se_{\wtau{2}}(0, \dx)}{(0, k(\wxf))} = %%
 \dueforma{\cH\se_{\wtau{2}}\iota \iota^{-1}(0, \dx)}{\iota \iota^{-1}(0, k(\wxf))} = \\ %%
& =  \dueforma{\cH\se_{\wtau{2}}\iota \ud\, (-\wtc)_* \dx}{\iota \ud\, (-\wtc)_* k(\wxf)} = %%
  \dueforma{\iota \ud\, (-\wtc)_* k(\wxf)}{\iota^{-1}\cH\se_{\wtau{2}}\iota \ud\, (-\wtc)_* \dx} = \\ %%
& =  \dueforma{\ud\, (-\wtc)_* k(\wxf)}{ {\wh{\cF}_{\wtau{2}\, *}}^{-1}{{\cH}_{\wtau{2}\, *}}    \ud\, (-\wtc)_* \dx} =  %%
\dueforma{\ud\, \left( -\wtc \circ \wSinv{\wtau{2}}\right)_*  \wt k(\wxd)}{ \ud\a_{2\, *}\wt k(\wxd)} .
\end{split}
\]
The last equality holds because $\Fref{t}$ is the lift of a vector field and thanks to \eqref{eq:alpha2}.
Moreover 
\[
\liededo{k}{(-\wtc)}{\wxf} = \liededo{\wt k}{(-\wtc \circ \wSinv{\wtau{2}})}{\wxd} .
\]
Substituting in \eqref{eq:posperp2} we finally get
\begin{equation*}
\begin{split}
0 <  &  H_{12}(\lu) - \liededo{\wt k}{(-\wtc \circ \wSinv{\wtau{2}})}{\wxd} + \dueforma{\ud\, \left( -\wtc \circ \wSinv{\wtau{2}}\right)_*  \wt k(\wxd)}{ \ud\a_{2\, *}\wt k(\wxd)}  = \\ %%
& = H_{12}(\lu) - \liededo{\wt k}{\a_2}{\wxd} ,
%%% \label{eq:posperp3}
\end{split}
\end{equation*}
a contradiction to \eqref{eq:contra}.

\subsection{Examples}\label{sec:examples}
\paragraph{Van der Pol Oscillator.}
As an example consider the following Van der Pol Oscillator, studied in \cite{Mau07} where the author numerically shows that the optimal control is bang-bang-singular. 
\begin{subequations}
\begin{align}
& \text{minimize } \ \dfrac{1}{2}\int_0^4 \left(\xi_1^2 + \xi_2^2 \right)(t)\ud t \quad \text{subject to } \\
\begin{split}
& \dot\xi_1(t) = \xi_2(t), \\
& \dot\xi_2(t) = -\xi_1(t) + \xi_2(t)\left( 1 - \xi_1^2(t) \right) + u(t) , \qquad 
\end{split}
\quad \qo t \in [0,4], \label{eq:dinamicavan2}\\
& \xi(0) = \left(0, 1\right), \quad \xi(4) \in \R^2. \label{eq:iniziovan2}
\end{align}
\end{subequations}
The problem can be restated as a Mayer problem by substituting the state variable $\xi$ with a state variable (which we still denote as $\xi$) in $\R^3$:
\begin{subequations}
\begin{align}
& \text{minimize } \ \xi_3(4) \quad \text{subject to } \\
\begin{split}
& \dot\xi_1(t) = \xi_2(t), \\
& \dot\xi_2(t) = -\xi_1(t) + \xi_2(t)\left( 1 - \xi_1^2(t) \right) + u(t) , \qquad \\
& \dot\xi_3(t) = \dfrac{1}{2}\left(\xi_1^2(t) + \xi_2^2(t) \right) , 
\end{split}
\quad \qo t \in [0,4], \label{eq:dinamicavan}\\
& \xi(0) = \left(0, 1, 0\right), \quad \xi(4) \in \R^3. \label{eq:iniziovan}
\end{align}
\end{subequations} 
More precisely the author numerically shows that the optimal control has two bang arcs and a singular arc where the control can be written as a feedback control.
\[
\wh u = \begin{cases}
-1  \quad & t\in [0,  \wtau{1}), \\
1  \quad & t\in (\wtau{1}, \wtau{2}), \\
u_{\rm sing}(x) =  2 x_1 - x_2 \left( 1 - x_1^ 2 \right) \quad & t\in (\wtau{2}, 4].
\end{cases}
\]
with $\wtau{1}\simeq 1.3667$, $\wtau{2} \simeq 2.4601$.

The problem fits in our setting defining 
\[
h_1(x) = h_3(x) = \begin{pmatrix}
x_2 \\
-x_1 + x_2(1-x_1^2)- 1\\
\frac{x_1^ 2+ x_2^2}{2}
\end{pmatrix}, \quad 
h_2(x) = \begin{pmatrix}
x_2 \\
-x_1 + x_2(1-x_1^2)+ 1\\
\frac{x_1^ 2+ x_2^2}{2}, 
\end{pmatrix},
\]
\[
\mathcal{X} = \co\left\{ h_1, h_2 \right\}, \quad 
\wu(x) = \dfrac{1+ u_{\rm sing}(x)}{2} \in (0,1), \quad
\fS(x) = \begin{pmatrix}
0 \\-2\\0
\end{pmatrix}.
\]

\paragraph{Bilinear systems.} Consider the following example proposed in \cite{LS15} with state-space $M := \left\{
N = \left( N_1, \ldots N_n \right) \in \Rn \colon  N_i > 0, \ i =1, \ldots n
\right\}$ and control set $U :=\displaystyle\prod_{i=1}^m [0, u_i^{\max}] $:
\begin{align*}
& \text{minimize } \ C(u) := \scal{r}{N(T)} + \int_0^ T \scal{q}{N(T)} + \scal{s}{u(t)} \ud t  \ \text{ subject to } \\
& \dot N(t) = \left( A + \sum_{j=1}^{m} u_j(t) B_j \right)N(t)
\quad \qo t \in [0,T], \\
& u \in L^\infty\left([0, T], U \right), \\\
& N(0) = N_0 .
\end{align*}
where $T >0$ is fixed and $A, \ B_1, \ldots, B_m$ are given $n \times n$ matrices, % $q \in (\Rn)^*$ and $s \in (\R^m)^*$ are given row vectors.

The problem can be transformed into a Mayer one on $M \times \R$ and the control box can be normalised to the unit control box $\wt U := [0, 1]^m $ by setting 
\[
\wt s_j = u_j^{\max} s_j , \quad C_j := u_j^{\max} B_j, \quad \forall j=1, \ldots m
\]
as
\begin{align*}
& \text{minimize } \ C(u) := \scal{r}{N(T)} + N_{n+1}(T)   \ \text{ subject to } \\
& \dot N(t) = \left( A + \sum_{j=1}^{m} u_j(t) C_j \right)N(t)
\quad \qo t \in [0,T], \\
& \dot N_{n+1}(t) =  \scal{q}{N(t)} + \scal{\wt s}{u(t)}, \\
& u \in L^\infty\left([0, T], \wt U \right), \\\
& N(0) = N_0 , \quad N_{n+1}(0) = 0.
\end{align*}
Denote as $\wt x = \left(x, x_{n+1} \right)$ the points in $M \times \Rn$ and set
\[
\wt r := \begin{pmatrix}
r, 1
\end{pmatrix}, \
f_0 (\wt x) := \begin{pmatrix}
A & 0 \\
q & 0
\end{pmatrix}\wt x, \quad
f_j(\wt x) := \begin{pmatrix}
C_j & 0 \\
0 & 1
\end{pmatrix}\begin{pmatrix}
x \\ \wt s_j
\end{pmatrix} \quad j=1, \ldots, m.
\]
Then the problem can be written as 
\begin{align*}
& \text{minimize } \ C(u) := \scal{\wt r}{\wt\xi(T)}\ \text{ subject to } \\
& \dot{ \wt\xi}(t) = f_0(\wt\xi(t))  + \sum_{j=1}^{m} u_j(t) f_j(\wt\xi(t))
\quad \qo t \in [0,T], \\
&  u \in L^\infty\left([0, T], \wt U \right), \\\
& \wt\xi(0) = \begin{pmatrix}
N_0 \\ 0
\end{pmatrix}.
\end{align*}
Thus the problem fits into our setting defining $X_1 = f_0$, $X_{j+1} = f_0 + f_j$, $j=1, \ldots, m$ $X_{m - 1 + j+ k } = f_0 + f_j + f_k$, $1 \leq j < k \leq m$, \ldots ,  $X_{2^m} = f_0 + f_1 + f_2 +\ldots + f_m$.
%%%%%%%%%%%%%%%%%%%%%%%%%%%%%%%%%%%%%%%%%%%%%
\bibliography{bibliocompleta} \bibliographystyle{plain}
\end{document}